\documentclass[10pt,a4paper]{amsart}
\usepackage[colorlinks=true, pdfstartview=FitV, linkcolor=blue, 
citecolor=blue]{hyperref}

\usepackage{amsthm,amsmath,amssymb,graphicx,comment}

\usepackage{amssymb,bbm,amscd}
\usepackage{microtype,xcolor}
\usepackage{a4wide}

\theoremstyle{plain}
\newtheorem{theorem}{Theorem}[section]
\newtheorem{proposition}[theorem]{Proposition}
\newtheorem{lemma}[theorem]{Lemma}

\newtheorem{remark}[theorem]{Remark}

\theoremstyle{definition}
\newtheorem{definition}[theorem]{Definition} 

\numberwithin{equation}{section}

\makeatletter
  \def\vhrulefill#1{\leavevmode\leaders\hrule\@height#1\hfill \kern\z@}
\makeatother

\newcommand{\ts}{\hspace{0.5pt}}
\newcommand{\nts}{\hspace{-0.5pt}}

\newcommand{\exend}{\hfill $\Diamond$}

\newcommand{\myfrac}[2]{\frac{\raisebox{-2pt}{$#1$}}
      {\raisebox{0.5pt}{$#2$}}}

\begin{document}

\title[Autocorrelations of the Thue--Morse sequence]{On the absolute value of the autocorrelations\\[1mm] of the Thue--Morse sequence}

\author{Michael Coons}
\author{Jan Maz\'a\v{c}}
\author{Ari Pincus--Kazmar}
\author{Adam Stout}

\address{Department of Mathematics and Statistics, California State
  University\newline \hspace*{\parindent}400 West First
  Street, Chico, California 95929, USA} \email{mjcoons@csuchico.edu, astout@csuchico.edu}
\address{Fakult\"at f\"ur Mathematik, Universit\"at Bielefeld, \newline
\hspace*{\parindent}Postfach 100131, 33501 Bielefeld, Germany}
\email{jmazac@math.uni-bielefeld.de}
\address{Department of Mathematics, Statistics and Computer Science, Macalester
  College\newline \hspace*{\parindent}1600 Grand Avenue, 
  St.~Paul, Minnesota 55105, USA} \email{aripincuskazmar@gmail.com}

\makeatletter
\@namedef{subjclassname@2020}{%
  \textup{2020} Mathematics Subject Classification}
\makeatother



 \date{\today}

\begin{abstract} Recently, Baake and Coons proved several results on the average size of the autocorrelations of the Thue--Morse sequence. They also considered the absolute value of the autocorrelations, and showed that the average value of the autocorrelations is zero. In particular, they showed that $\sum_{n\leqslant x}|\eta(n)|=o(x^\alpha)$ for any $\alpha>\log(3)/\log(4)$. In this paper, we sharpen this result, providing upper and lower bounds for $\alpha$. On the way to our lower bounds, we obtain the structure of the linear representation of the point-wise product of two $k$-regular sequences, which may be of independent interest.
\end{abstract}

\maketitle

\section{Introduction}

Let $t$ be the (one-sided) Thue--Morse sequence \cite{M21,Thue},
or word, taking the values $\pm 1$, defined by $t(0)=1$ and, for
$k \geqslant 0$, by $t(2k)=t(k)$ and $t(2k+1)=-t(k)$. The standard autocorrelation coefficients
$\eta(m)$, for each $m\geqslant 0$, of the Thue--Morse sequence, are given by
\[
  \eta(m) \, = \lim_{N\to\infty}\myfrac{1}{N}
  \sum_{k=0}^{N\nts -1} t(k) \ts t(k+m) \ts .
\] Since the discrete hull of the (two-sided) Thue--Morse sequence is uniquely ergodic under the $\mathbb{Z}$-action of the shift, an application of Birkhoff's ergodic theorem guarantees the uniform existence of its autocorrelation measure, which then establishes the existence of the limits $\eta(m)$, see Baake and Grimm \cite[Ch.~4]{TAO} for more details and background. The recurrences satisfied by the Thue--Morse sequence induce similar recurrences for the sequence $\eta$. In particular, one easily computes that $\eta(0)=1$, and, for
$m\geqslant 0$, using the Thue--Morse recursions, that
\begin{equation}\label{eq:v2recs}
\begin{split}
  \eta(2m) \, & = \, \eta(m) \qquad\mbox{and} \\
  \eta(2m+1) \, & = -\myfrac{1}{2}\big(\eta(m)+\eta(m+1)\big) .
\end{split}  
\end{equation}

These recursions, which go back to Mahler \cite{Mah27}, have many consequences as we shall see later. For example, one can immediately obtain certain bounds on the modulus of $\eta$. One has 
\[\max_{m\neq 0}\, \bigl|\eta(m) \bigr| \, = \, \myfrac{1}{3} \qquad \mbox{and} \qquad \bigl| \eta(2m+1) \bigr| < \myfrac{1}{3} \quad \mbox{for all} \ m\geqslant 2.  \]

One can also derive explicit values of $\eta$ at certain positions showing that the sequence contains infinitely many positive and negative elements, as well as zeros. It easily follows that, for all $n \geqslant 2$, we have 
\[\eta(2^n) \, = \, -\myfrac{1}{3}, \qquad \eta(2^n+2^{n-1}) \, = \, \myfrac{1}{3} \qquad \mbox{and} \qquad \eta(2^n+2^{n-2}) \, = \, 0.  \]

This immediately raises a question about the average behaviour of $\eta$. Recently, Baake and Coons~\cite{BC2024} showed that the $k$-th moments of $\eta$ and $|\eta|$ --- that is, the mean values of $\eta^k$ and $|\eta|^k$ --- are equal to zero for any integer $k\geqslant 1$. The case $k=1$ is of particular interest; the sign changes in $\eta$ contribute to such extreme cancellation that the partial sums of $\eta$ are uniformly bounded. In an attempt to better understand the contributions from the sign changes and the magnitudes of $\eta$, Baake and Coons also showed that for any $\alpha>\log_4(3)\approx 0.79241505$, \begin{equation}\label{eq:limit}\lim_{x\to\infty}\frac{1}{x^\alpha}\sum_{m\leqslant x}|\eta(m)|=0,\end{equation} noting that their bound for $\alpha$ is not optimal and a computation proves that the optimal value of $\alpha$ is bounded above by $0.652633$. Here, and throughout, for any integer $k\geqslant 2$, we denote the base-$k$ logarithm by $\log_k$.

Considering lower bounds for the optimal value of $\alpha$, by comparing with the growth of $\eta^2$, a~result of Zaks, Pikovsky, and Kurths \cite{ZPK1997} implies that $\alpha$ must be larger than $\log_2((1+\sqrt{17})/4)\approx0.357018636$.

In this paper, we provide improved bounds for $\alpha$. 

\begin{theorem}\label{thm:main} There exists positive constants $c_1$ and $c_2$, such that, for any $x$, \[c_1x^{0.6274882485}\leqslant \sum_{m\leqslant x}|\eta(m)|\leqslant c_2x^{0.6464616661}.\]
\end{theorem}

This paper is organized as follows. In Section \ref{sec:lower}, we provide a result on the linear representation of the point-wise product of regular sequences, and we examine the consequences of twisting $\eta$ by various automatic sequences of absolute value one, which results in proving the lower bound in Theorem \ref{thm:main}. In Section \ref{sec:upper}, we compute upper bounds for $\alpha$ and use this to prove the upper bound in Theorem~\ref{thm:main}. Finally, in Section \ref{sec:further}, we give some ideas for further research and discuss the optimal value of $\alpha$.

\section{Regular sequences and their growth}\label{sec:kreg}

Let $f:\mathbb{Z}_{\geqslant 0}\to\mathbb{R}$ be a sequence. We define the $k$-kernel of $f$ by \[{\rm ker}_k(f):=\{(f(k^\ell n+r))_{n\geqslant 0}:\ell\geqslant 0, 0\leqslant r<k^\ell\}.\] If ${\rm ker}_k(f)$ is finite, we say that the sequence $f$ is {\em $k$-automatic} \cite{C1972}. From above, we can see that the only two elements of the $2$-kernel of the Thue--Morse sequence are $t$ and $-t$, so we have that $|{\rm ker}_2(t)|=2$, and that $t$ is $2$-automatic. Now, in the case of $\eta$, it is quite clear, since $\eta$ takes infinitely many values, that ${\rm ker}_2(\eta)$ is infinite. But, the recurrences in \eqref{eq:v2recs} imply that ${\rm ker}_2(\eta)$ generates a~finitely generated $\mathbb{R}$-vector space of dimension $2$. In general, if the $\mathbb{R}$-vector space ${\rm Span}_\mathbb{R}\big({\rm ker}_k(f)\big)$ is finite dimensional, we say that the sequence $f$ is {\em $k$-regular}. The set of $k$-regular sequences was introduced by Allouche and Shallit \cite{AS1992} as a natural analogue of $k$-automatic sequences. Indeed, this concept is robust---a $k$-regular sequence that takes on only finitely many values is $k$-automatic \cite[Theorem 2.3]{AS1992}. If $f$ is $k$-regular, one can use a basis for ${\rm Span}_\mathbb{R}\big({\rm ker}_k(f)\big)$ to show that there is an integer $d\geqslant 1$, a real $1\times d$ vector $\boldsymbol{\ell}$, and real $d\times d$ matrices ${\bf A}_0,{\bf A}_1,\ldots, {\bf A}_{k-1}$ such that \[f(n)=\boldsymbol{\ell}\,{\bf A}_{i_{s}}{\bf A}_{i_{s-1}}\cdots{\bf A}_{i_{1}}{\bf A}_{i_{0}}{\bf e}_1,\] where the base-$k$ expansion of $n$ is $(n)_k=i_s i_{s-1}\cdots i_1i_0$ and ${\bf e}_1$ is the standard column basis vector of $\mathbb{R}^d$ with a $1$ in the first entry and zeros elsewhere. We call the collection \[L_k(f):=(\boldsymbol{\ell}, {\bf A}_0, {\bf A}_1,\ldots,{\bf A}_{k-1})\] the {\em linear representation} of the $k$-regular sequence $f$. Set ${\bf A}:={\bf A}_0+{\bf A}_1+\cdots+{\bf A}_{k-1}$, and let $\rho:=\rho({\bf A})$ denote the spectral radius of ${\bf A}$ and $\rho^*:=\rho^*({\bf A}_0, {\bf A}_1,\ldots,{\bf A}_{k-1})$ denote the joint spectral radius of the matrices ${\bf A}_0, {\bf A}_1,\ldots,{\bf A}_{k-1}$. 

The structure of regular sequences, and in particular, the existence of their linear representation, allowed Dumas \cite[Theorem 3]{D2013} to give an explicit description of the growth of their partial sums. 

\begin{proposition}[Dumas, 2013]\label{prop:Dumas} Suppose that the spectral radius $\rho=\rho({\bf A})$ is the unique dominant eigenvalue of ${\bf A}$, that $\rho$ has equal geometric and algebraic multiplicities, and that $\rho({\bf A})>\rho^*({\bf A}_0,\ldots,{\bf A}_{k-1})$. Then, there exists a H\"older continuous 1-periodic function $\psi$ that is bounded away from zero such that \[\sum_{m\leqslant x}f(m) =\psi\big(\log_k(x)\big)\,x^{\log_k(\rho)}+o\big(x^{\log_k(\rho)}\big).\] 
\end{proposition}

\section{Lower Bounds}\label{sec:lower}

The absolute value of a $k$-regular sequence is not necessarily $k$-regular. To circumvent this and to establish a lower bound for the partial sums of $|\eta|$, we exploit the fact that the product of two $k$-regular sequences is also $k$-regular \cite[Thm.~2.5]{AS1992}. In particular, we consider multiplication by various $2$-regular sequences $f$ whose absolute value is bounded by one. Then, obtaining the asymptotics of $\sum_{m\leqslant x}f(m)\eta(m)$ allows us to find lower bounds on the asymptotics of $|\eta|$ since \[\Bigg|\sum_{m\leqslant x}f(m)\eta(m)\,\Bigg|\leqslant \sum_{m\leqslant x}\big|\eta(m)\big|.\]

As a first step toward these lower bounds, we considered $f=\eta$, so that $f(m)\eta(m)=\eta^2(m)$, which is $2$-regular. Using Dumas' method described in the previous section, we rediscovered the result of Zaks, Pikovsky, and Kurths \cite{ZPK1997} that $x^{-\log_4(\alpha)}\sum_{m\leqslant x}\eta^2(m)$ is eventually bounded between two positive constants, where $\alpha=(1+\sqrt{17})/4$ and $\log_2(\alpha)\approx 0.357018636$. Using these methods, we could also give \emph{explicit} formulas for certain values of $x$. If $f(m) = \eta(m)$, we obtain for $x=2^n-1$ 
\[ \sum_{m\leqslant 2^n-1} \eta^2(m) \, = \, \myfrac{85+19\sqrt{17}}{306} \left(\myfrac{1+\sqrt{17}}{4}\right)^n + \myfrac{85-19\sqrt{17}}{306} \left(\myfrac{1-\sqrt{17}}{4}\right)^n + \myfrac{4}{9}.\]
Analogously, for $f(m) = \eta(m+1)$, we can derive 
\[ \sum_{m\leqslant 2^n-1} \eta(m)\eta(m+1) \, = \, -\myfrac{51+5\sqrt{17}}{306} \left(\myfrac{1+\sqrt{17}}{4}\right)^n + \myfrac{-51+5\sqrt{17}}{306} \left(\myfrac{1-\sqrt{17}}{4}\right)^n.\]

We tested other functions $f$ as well. Figure \ref{fig:compare3} shows the partial sums of $f\eta$ for  three paradigmatic examples of $2$-automatic functions $f$ that satisfy $f\in\{-1,1\}$, in particular, for the paperfolding sequence \cite[p.~155]{ASbook}, the Rudin--Shapiro sequence \cite[p.~78]{ASbook}, and the Thue--Morse sequence. The inherent growth displayed in Figure \ref{fig:compare3} (right) suggests that one should more closely examine $f=t$, the Thue--Morse sequence.

\begin{figure}[ht]
\begin{center}
  \includegraphics[width=.32\linewidth]{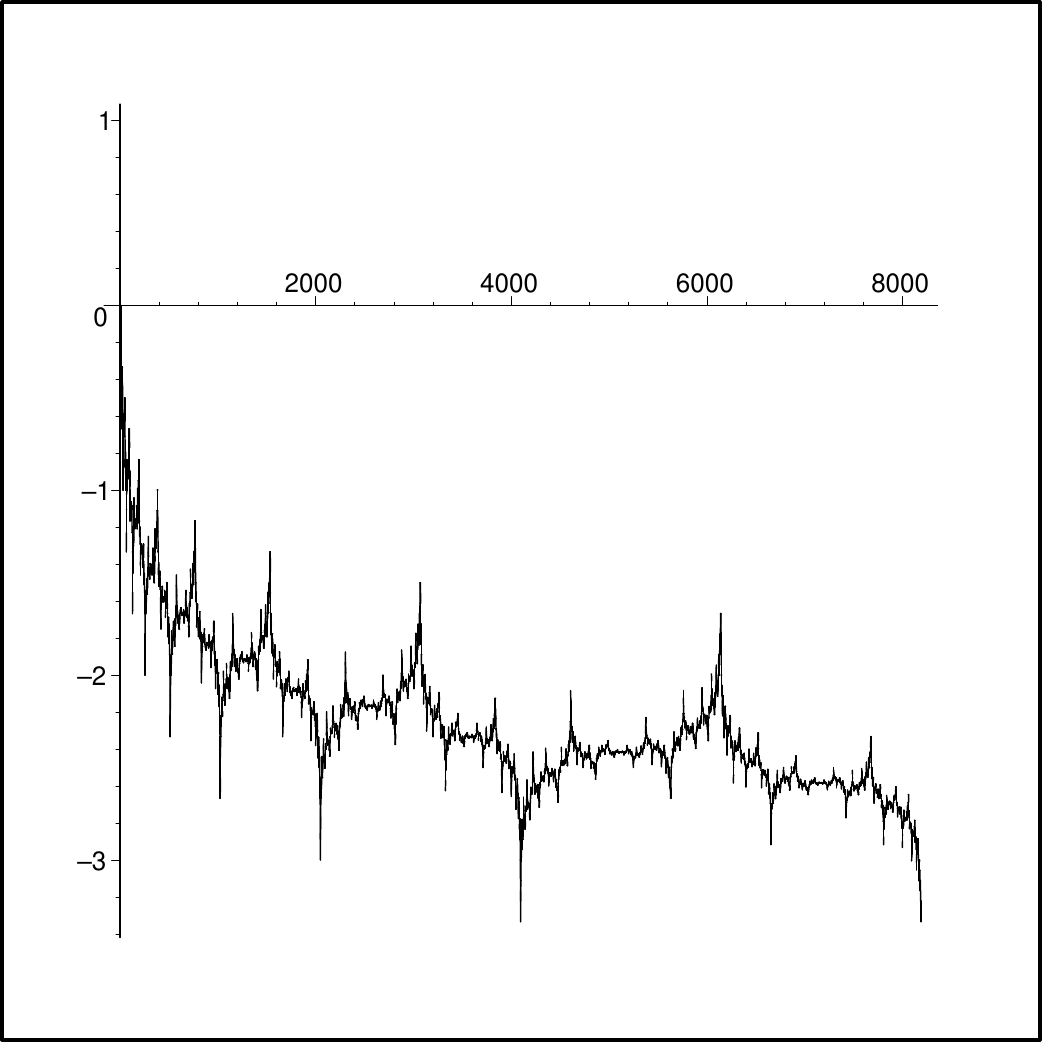}
  \includegraphics[width=.32\linewidth]{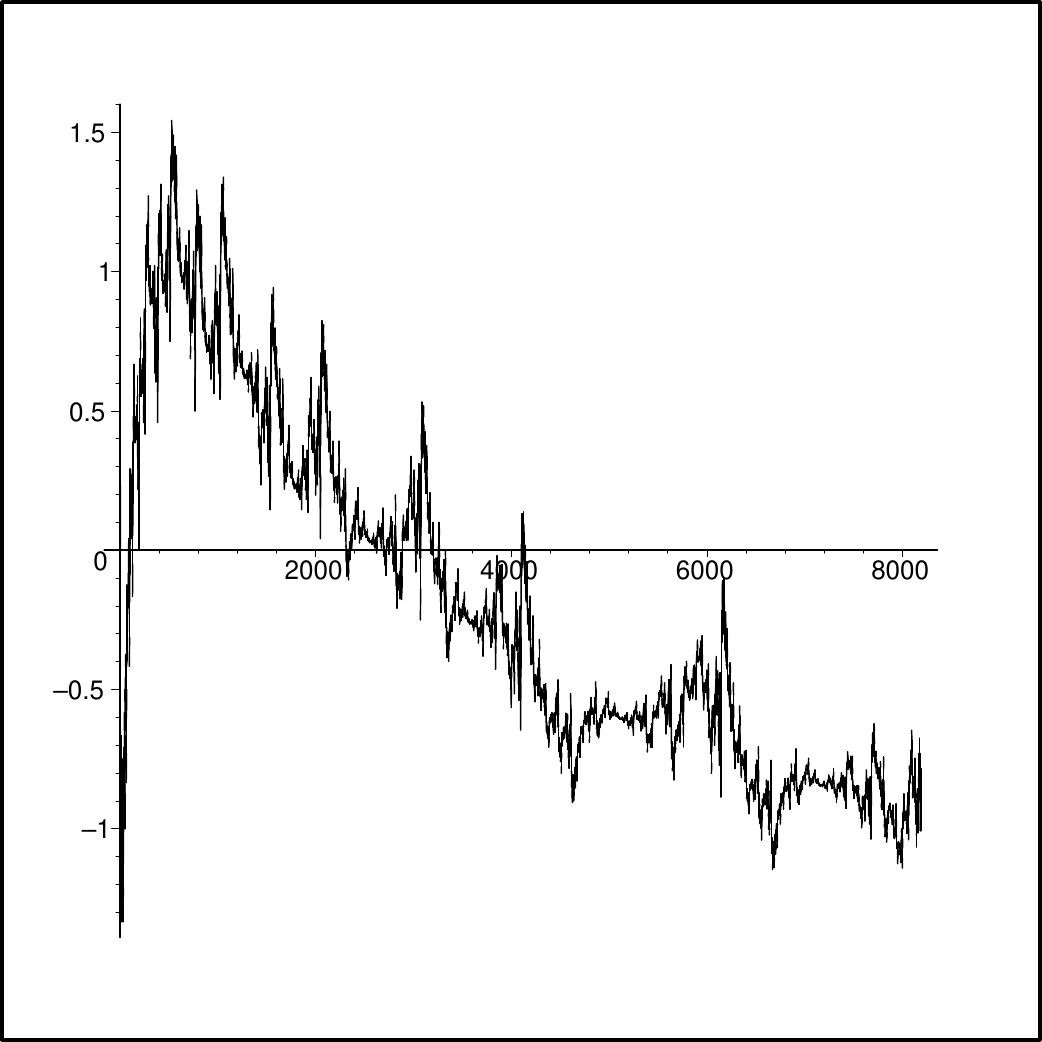}
  \includegraphics[width=.32\linewidth]{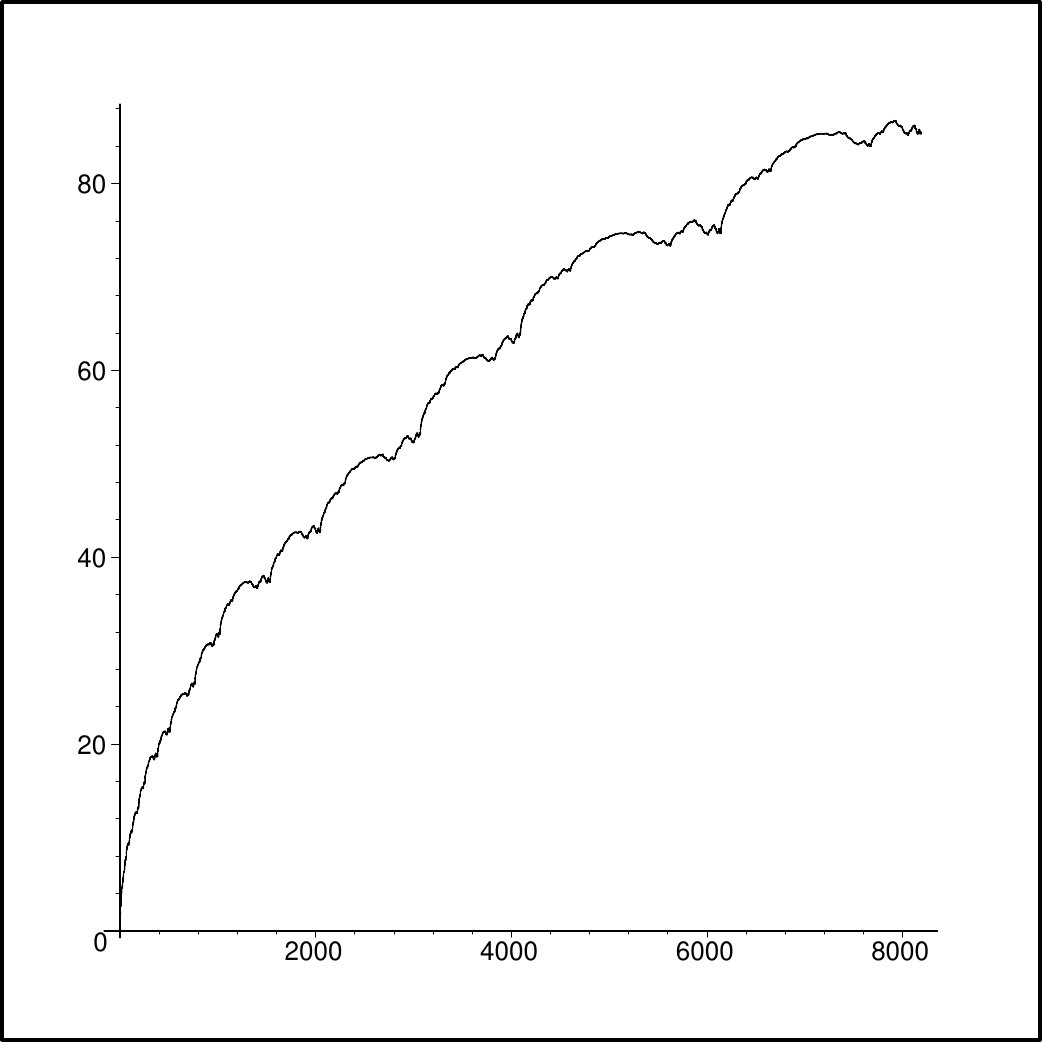}
\end{center}
\caption{The values of $\sum_{m\leqslant x}f(m)\eta(m)$ for $f$ the paperfolding sequence (left), the Rudin--Shapiro sequence (middle), and the Thue--Morse sequence (right), for values of $x=0,\ldots,2^{13}-1$.}
\label{fig:compare3}
\end{figure}

\begin{proposition}\label{prop:lowerbound} There exists a positive constant $c_1$ such that, for all $x$, \[c_1\sqrt{x}\leqslant \sum_{m\leqslant x}|\eta(m)|.\]
\end{proposition}

\begin{proof} We consider the sequence $t(m)\eta(m)$ and show that the partial sums of this sequence are asymptotic to a multiple of $\sqrt{x}$. To see that the result follows from this, we note that \[c_1\sqrt{x}\leqslant  \Bigg|\sum_{m\leqslant x}t(m)\eta(m)\Bigg|\leqslant  \sum_{m\leqslant x}|\eta(m)|,\] since $|t(m)|=1$.

To prove this, we will apply Dumas' method \cite{D2013} to the linear representation of the $2$-regular sequence $t(m)\eta(m)$. We find the linear representation via the recursions \[t(2m)\eta(2m)=t(m)\eta(m),\] and \[t(2m+1)\eta(2m+1)=\frac{t(m)}{2}(\eta(m)+\eta(m+1))=\frac{1}{2}t(m)\eta(m)+\frac{1}{2}t(m)\eta(m+1).\] These recurrences, along with the recurrences for $\eta(m)$, imply the matrix recurrences \[\big(t(2m)\eta(2m)\ \ t(2m)\eta(2m+1)\big)=\big(t(m)\eta(m)\ \ t(m)\eta(m+1)\big)\begin{pmatrix}1&-\frac{1}{2}\\ 0 & -\frac{1}{2}\end{pmatrix},\] and \[\big(t(2m+1)\eta(2m+1)\ \ t(2m+1)\eta(2m+2)\big)=\big(t(m)\eta(m)\ \ t(m)\eta(m+1)\big)\begin{pmatrix}\frac{1}{2}&0\\ \frac{1}{2} & -1\end{pmatrix},\] which further imply that the $2$-kernel of $t(m)\eta(m)$ generates a two-dimensional real vector space generated by the two sequences $t(m)\eta(m)$ and $t(m)\eta(m+1)$ Thus, $t(m)\eta(m)$ has linear representation \[L_2(f)=\left(\boldsymbol{\ell}=\begin{pmatrix}1& -\frac{1}{3}\end{pmatrix},\ {\bf A}_0=\begin{pmatrix}1&-\frac{1}{2}\\ 0 & -\frac{1}{2}\end{pmatrix},\ {\bf A}_1=\begin{pmatrix}\frac{1}{2}&0\\ \frac{1}{2} & -1\end{pmatrix}\right).\]

Since ${\bf A}={\bf A}_0+{\bf A}_1$ is traceless, it does not have a unique dominant eigenvalue, so we cannot apply Dumas' result directly to the above linear representation. That said, viewed as a $4$-regular sequence, $f$ has the linear representation \[L_4(f)=\left(\boldsymbol{\ell}=\begin{pmatrix}1& -\frac{1}{3}\end{pmatrix},\ {\bf B}_0={\bf A}_0^2,\ {\bf B}_1={\bf A}_0{\bf A}_1,\ {\bf B}_2={\bf A}_1{\bf A}_0,\ {\bf B}_3={\bf A}_1^2 \right).\] Moreover, we have that $\rho({\bf B})=\rho({\bf A}^2)=2$ is the unique dominant eigenvalue of ${\bf B}={\bf A}^2=2{\bf I}$ and has geometric and algebraic multiplicity equal to $2$. Furthermore, we have $\rho^*({\bf B}_0,{\bf B}_1,{\bf B}_2,{\bf B}_3)=1$ since each of the matrices has maximum column sum norm equal to $1$, and $1$ is an eigenvalue of ${\bf B}_0$ and ${\bf B}_3$. 

Thus Dumas' result, Proposition~\ref{prop:Dumas}, applied to $f=t\cdot\eta$ as a $4$-regular sequence, gives that \[x^{-\log_4(2)}\Bigg|\sum_{m\leqslant x}t(m)\eta(m)\Bigg|=x^{-1/2}\Bigg|\sum_{m\leqslant x}t(m)\eta(m)\Bigg|\] is eventually bounded between two positive numbers. This proves the result.
\end{proof}

\begin{figure}[ht]
\begin{center}
  \includegraphics[width=.9\linewidth]{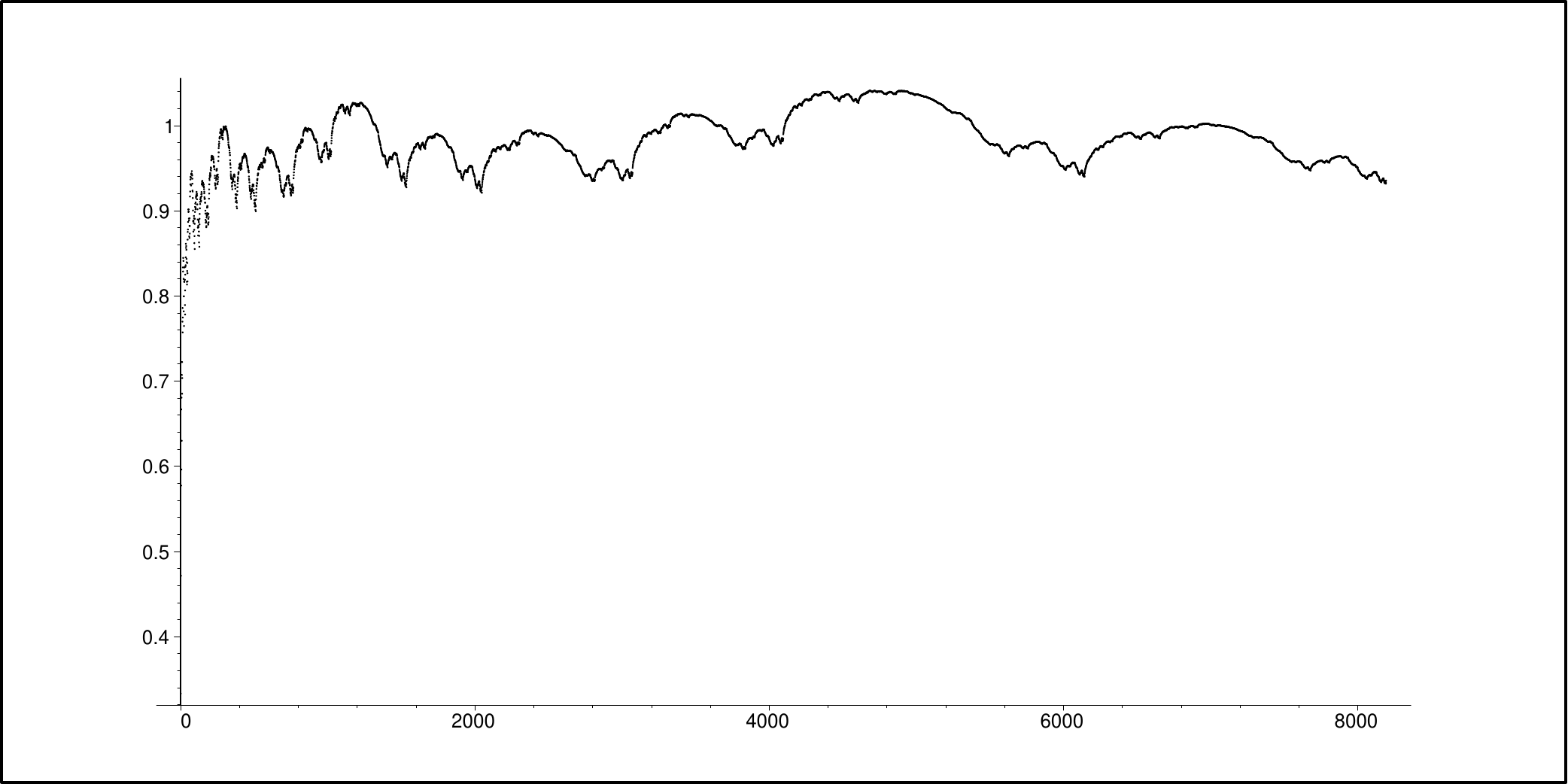}
\end{center}
\caption{The values of $x^{-1/2}\sum_{m\leqslant x}t(m)\eta(m)$ for  $x=0,\ldots,2^{13}$.}
\label{fig:TMetasum0to13overSqrt}
\end{figure}

Figure \ref{fig:TMetasum0to13overSqrt} contains a plot of the function $x^{-1/2}\sum_{m\leqslant x}t(m)\eta(m)$ for $x$ up to $2^{13}$. 

The form of the matrices ${\bf A}_0$ and ${\bf A}_1$ in the preceding proof suggests that the linear representation of a point-wise product of $k$-regular sequences is given by the Kronecker (tensor) product of their respective linear representations. It turns out this is precisely the case.

\begin{proposition}\label{prop:kronecker} Let $f$ and $g$ be $k$-regular sequences with linear representations \[L_k(f):=(\boldsymbol{\ell}, {\bf A}_0, {\bf A}_1,\ldots,{\bf A}_{k-1})\quad \mbox{and}\quad L_k(g):=(\boldsymbol{m}, {\bf B}_0, {\bf B}_1,\ldots,{\bf B}_{k-1}).\] Then \[L_k(fg):=(\boldsymbol{\ell}\otimes\boldsymbol{m}, {\bf A}_0\otimes{\bf B}_0, {\bf A}_1\otimes{\bf B}_1,\ldots,{\bf A}_{k-1}\otimes{\bf B}_{k-1}).\]
\end{proposition}

\begin{proof} Without loss of generality, we assume that the linear representations $L_k(f)$ and $L_k(g)$ have been derived from bases $\{f=f_1,\ldots,f_r\}$ and $\{g=g_1,\ldots,g_s\}$ of the vector spaces spanned by the $k$-kernels of $f$ and $g$, respectively. We use the basis \begin{equation}\label{eq:fgbasis}\{fg=f_1g_1,\ldots,f_1g_s,f_2g_1,\ldots,f_2g_s,\ldots,f_rg_s\},\end{equation} of the $k$-kernel of $fg$, where this basis is ordered lexicographically in the indices $(i,j)$ of $f_ig_j$. It is clear that the initial vector then is $\boldsymbol{\ell}\otimes\boldsymbol{m}$. Now, for each $b\in\{0,1,\ldots,k-1\}$, to find the $b$-th matrix, ${\bf C}_b$, in the linear representation (related to the basis in \eqref{eq:fgbasis}) of $fg$, we need only determine each row of ${\bf C}_b$. The entries of the $(s(i-1)+j)$-th row of ${\bf C}_b$ for $i\in\{1,\ldots,r\}$ and $j\in\{1,\ldots,s\}$ are the coefficients of the lexicographically ordered sequences in \eqref{eq:fgbasis}. These are determined by the equation \begin{align*}f_i(kn+b)g_j(kn+b)&=\left(\sum_{\ell=1}^r ({\bf A}_b)_{(i,\ell)}f_\ell(n)\right)\left(\sum_{m=1}^s ({\bf B}_b)_{(j,m)}g_m(n)\right)\\
&=\sum_{\ell=1}^r\sum_{m=1}^s ({\bf A}_b)_{(i,\ell)}({\bf B}_b)_{(j,m)}f_\ell(n)g_m(n)\\
&=\sum_{\ell=1}^r\sum_{m=1}^s ({\bf A}_b\otimes{\bf B}_b)_{(s(i-1)+j,s(\ell-1)+m)}f_\ell(n)g_m(n).
\end{align*} Thus ${\bf C}_b={\bf A}_b\otimes{\bf B}_b$, which finishes the proof.
\end{proof}

Note that in Proposition \ref{prop:kronecker}, we are working with a pair of specific linear representations. This does not create a problem, since the relevant quantities are the trace, determinant, and spectral radius of the sum matrices, which are invariant. Also, note that for any two sequences $f$ and $g$, we have $fg(n)=gf(n)$, but the Kronecker product is generally not commutative, though the products are conjugate via a permutation matrix. Using the reverse product will also change the order of the basis, sending $f_ig_j$ to $g_if_j$.

In the rest of this section, we exploit this result to make an increasing sequence of $2^n$-regular sequences $f^{(n)}$, taking values in $\{-1,1\}$, that approximate the signs of $\eta$. We prove the lower bound of Theorem \ref{thm:main} by determining the growth of the partial sums of $f^{(n)}\eta$ for $n=25$. 

To this end, recall that \[L_2(\eta)=\left(\boldsymbol{\ell}=\begin{pmatrix}1& -\frac{1}{3}\end{pmatrix},\ {\bf A}_0=\begin{pmatrix}1&-\frac{1}{2}\\ 0 & -\frac{1}{2}\end{pmatrix},\ {\bf A}_1=\begin{pmatrix}-\frac{1}{2}&0\\ -\frac{1}{2} & 1\end{pmatrix}\right),\] so that, 
\begin{align*} 
L_4(\eta)&=\left(\boldsymbol{\ell},{\bf A}_0{\bf A}_0,{\bf A}_0{\bf A}_1,{\bf A}_1{\bf A}_0,{\bf A}_1{\bf A}_1\right)\\
L_8(\eta)&=\left(\boldsymbol{\ell},
{\bf A}_0{\bf A}_0{\bf A}_0,
{\bf A}_0{\bf A}_0{\bf A}_1,
{\bf A}_0{\bf A}_1{\bf A}_0,
{\bf A}_0{\bf A}_1{\bf A}_1,
{\bf A}_1{\bf A}_0{\bf A}_0,
{\bf A}_1{\bf A}_0{\bf A}_1,
{\bf A}_1{\bf A}_1{\bf A}_0,
{\bf A}_1{\bf A}_1{\bf A}_1,
\right)
\end{align*} and, in general, \begin{equation}\label{eq:linrep2n}L_{2^n}(\eta)=\left(\boldsymbol{\ell},{\bf B}_0,{\bf B}_1,\ldots,{\bf B}_{2^n-1}\right),\end{equation} with ${\bf B}_b={\bf A}_{i_s}\cdots{\bf A}_{i_1}{\bf A}_{i_0}$, where the binary expansion of $b$ is $(b)_2=i_s\cdots i_1i_0$.

\begin{definition} For each $n\geqslant 1$, define the sequence $f^{(n)}$ as the $2^n$-regular sequences with linear representation \[L_{2^n}(f^{(n)})=\Big(\boldsymbol{\ell}=(1\ \ -1),\ {\bf D}_0,\ {\bf D}_1,\ldots,\ {\bf D}_{2^n-1}\Big),\] where each ${\bf D}_b$ is defined in terms of the entries of ${\bf B}_b$. In particular, if \[{\bf B}_b=\begin{pmatrix} b_{11} & b_{12}\\ b_{21} & b_{22}\end{pmatrix},\quad \mbox{then}\quad {\bf D}_b=\begin{pmatrix} d_{11} & d_{12}\\ d_{21} & d_{22}\end{pmatrix},\] where for each $i=1,2$, denoting ${\rm sgn}(x)$ as the sign of the real number $x$, \[ ( d_{i1}\ \ d_{i2}) =\begin{cases} 
({\rm sgn}(b_{i1})\ \ 0), & \mbox{if $|b_{i1}|>|b_{i2}|$};\\
(0\ \ {\rm sgn}(b_{i2})), & \mbox{if $|b_{i1}|\leqslant|b_{i2}|$}.\end{cases}\]
\end{definition}

\begin{lemma} For each $n\geqslant 2$, the sequence $f^{(n)}$ takes only the values $1$ or $-1$.
\end{lemma}

\begin{proof} The dimension of the matrices and form of the linear representation imply that the $2^n$-kernel is generated by at most $2$ sequences $f^{(n)}=f^{(n)}_1$ and $f^{(n)}_2$, where the definition of $L_{2^n}(f^{(n)})$ gives, for $(m)_{2^n}=i_si_{s-1}\cdots i_0$, that $f_i^{(n)}=\boldsymbol{\ell}\,{\bf D}_{i_s}{\bf D}_{i_{s-1}}\cdots {\bf D}_{i_0}{\bf e}_i$ for $i\in\{1,2\}$. Then, for each $b\in\{0,1,\ldots,2^n-1\}$, we have for any value $m\geqslant 0$ that \[f^{(n)}(2^nm+b)=\begin{cases}{\rm sgn}(b_{i1})\cdot f^{(n)}(m)  & \mbox{if $|b_{i1}|>|b_{i2}|$};\\
{\rm sgn}(b_{i2})\cdot f^{(n)}_2(m), & \mbox{if $|b_{i1}|\leqslant|b_{i2}|$}.\end{cases}\] The lemma follows since the initial values of $f^{(n)}$ and $f^{(n)}_2$ are $1$ and $-1$, respectively.
\end{proof}

\begin{proof}[Proof of lower bound in Theorem \ref{thm:main}] We prove this result via a computation. For each integer $n\in\{3,4,\ldots,25\}$, we form $L_{2^n}(f^{(n)}\eta)$ using Proposition \ref{prop:kronecker}. We then check that the $4\times 4$ matrix \[{\bf A}:=\sum_{b=0}^{2^n-1}{\bf D}_b\otimes {\bf B}_b\] has a unique simple dominant eigenvalue, and that $\rho({\bf A})>2$. Here, the joint spectral radius of the matrices ${\bf D}_b\otimes {\bf B}_b$ is bounded by $2$. To see this, note that  the matrices ${\bf B}_b$ each have absolute column sums bounded by $1$, and so, since ${\bf D}_b$ is a matrix with entries absolutely bounded by $1$, the absolute column sums of ${\bf D}_b\otimes {\bf B}_b$ can be at most twice that of ${\bf B}_b$. With the conditions of Proposition \ref{prop:Dumas} met, we can apply the result. Table \ref{table:lowers} contains the data from our computations that gives the result.
\end{proof}

\begin{table}[h]
{\footnotesize 
\begin{center}
\begin{tabular}{r|l|r|l}
$n$ & $\log_{2^n}(\rho({\bf A}))$ & $n$ & $\log_{2^n}(\rho({\bf A}))$	\\ \hline
3 & 0.565666799497928& 15& 0.623073449415594\\
4 & 0.596966648899461& 16& 0.623758591340177\\
5 & 0.604231518582344& 17& 0.624365750908173\\
6 & 0.608089435215993& 18& 0.624908958903671\\
7 & 0.610601912783904& 19& 0.625394632874860\\
8 & 0.613587830520361& 20& 0.625830845339636\\
9 & 0.615845308836204& 21& 0.626225163486055\\
10& 0.617611929153285& 22& 0.626583892017639\\
11& 0.619003323664352& 23& 0.626911602385536\\
12& 0.620255188737662& 24& 0.627211995904767\\
13& 0.621351234680364& 25& 0.627488248536345\\
14& 0.622287208010010& & \\
\end{tabular}
\end{center}
}
\vspace{0.1cm}
\caption{Values of $\log_{2^n}(\rho({\bf A}))$, where ${\bf A}:=\sum_{b=0}^{2^n-1}{\bf D}_b\otimes {\bf B}_b$, for $n=3,4,\ldots,25$.}
\label{table:lowers}
\end{table}

\begin{remark} The entries of the `sign' matrices ${\bf D}_b$ were chosen to mimic the total impact of the signs of the entries of ${\bf B}_b$. In this way, the functions $f^{(n)}$ mimic ${\rm sgn}(\eta)$, in a limiting sense. 

There are a few immediate ways that one can modify the definition of $f^{(n)}$ based on the comparison of the values $|b_{i1}|$ and $|b_{i2}|$. We tried each of these, and computed the associated values of $\log_{2^n}(\rho({\bf A}))$. Each of these was less than the tabulated values for $f^{(n)}$ as defined above. 

Additionally, one may want to know how often $f^{(n)}$ matches up with the signs of $\eta$; e.g., is there a growing initial segment (in $n$) where $f^{(n)}\eta$ is positive? We computed the initial values for $n\in\{1,\ldots,19\}$ and it turns out that, in this range, $f^{(n)}(5)=(-1)^n$. And there are other values with this property. So, if there is some initial growing agreement of these functions, then it would seem the function of $n$ describing the length of this segment is growing slowly.\exend
\end{remark}

\section{Upper Bounds}\label{sec:upper}

In this section, we prove the upper bound for Theorem \ref{thm:main} by exploiting generalizations of the recursions \eqref{eq:v2recs}. For example, the recursions \eqref{eq:v2recs} immediately imply that, when we consider the arithmetic progressions with common difference $4$, we have  
\begin{equation}\label{eq:v4recs}
\begin{split}
  \eta(4m) \, & = \, \eta(m) \\
  \eta(4m+1) \, & = -\myfrac{1}{4}\eta(m)+\myfrac{1}{4}\eta(m+1)\\
  \eta(4m+2) \, & = -\myfrac{1}{2}\eta(m)-\myfrac{1}{2}\eta(m+1) \\
  \eta(4m+3) \, & = \myfrac{1}{4}\eta(m)-\myfrac{1}{4}\eta(m+1) .
\end{split}  
\end{equation}
The absolute values of the coefficients on the righthand side of \eqref{eq:v4recs} add up to $3$. Baake and Coons \cite[Theorem 3.5]{BC2024} used this to show that the equality in \eqref{eq:limit} holds for any $\alpha>\log_4(3)$. Their proof immediately extends {\em mutatis mutandis} and proves the following result. 

\begin{lemma}\label{lem:upper} If ${\bf R}_n$ is the $2^n\times 2$ matrix such that \[\begin{pmatrix}\eta(2^n m)\\ \eta(2^n m+1)\\ \vdots\\ \eta(2^n m+2^n-1)\end{pmatrix}={\bf R}_n\begin{pmatrix}\eta(m)\\ \eta(m+1)\end{pmatrix},\] then Equation \eqref{eq:limit} holds for all $\alpha>\log_{2^n}\hspace{-0.1cm}\big(\|{\bf R}_n\|_{\rm abs}\big),$ where $\|{\bf R}_n\|_{\rm abs}$ denotes the sum of the absolute values of the entries of the matrix ${\bf R}_n$.
\end{lemma}

\begin{proof}[Proof of upper bound in Theorem \ref{thm:main}] We use Lemma \ref{lem:upper} and note that to compute the necessary sum, we need only \textcolor{red}{to} compute the sum of the absolute values of the entries of the first rows of the matrices ${\bf B}_b$ from the linear representation of $\eta$ in \eqref{eq:linrep2n}. See the data we produced in Table \ref{table:values}. The value corresponding to $n=30$ provides the upper bound, and finishes the proof of the theorem.
\end{proof}

\begin{remark} Note that the proof of Theorem \ref{thm:main} implies that Equation \ref{eq:limit} holds for any value of $\alpha>0.6464616660609581$.\exend
\end{remark}

\begin{table}[h]
{\footnotesize 
\begin{center}
\begin{tabular}{r|l|l|r|l|l}
$n$	& $\|{\bf R}_n\|_{\rm abs}$& $\log_{2^n}\hspace{-0.1cm}\big(\|{\bf R}_n\|_{\rm abs}\big)$ &$n$	& $\|{\bf R}_n\|_{\rm abs}$ & $\log_{2^n}\hspace{-0.1cm}\big(\|{\bf R}_n\|_{\rm abs}\big)$	\\ \hline
$1$	& $2$	& $1.0$	& $16$& $1464.43762207031$	& $0.6572581892116728$\\
$2$	& $3$	& $0.792482$& $17$& $2272.89575195312$	& $0.6559009294152249$	\\
$3$	& $5$	& $0.773977$& $18$& $3527.49731445312$	& $0.6546905149519707$	\\
$4$	& $7.5$	& $0.726723$& $19$& $5474.64619445801$	& $0.6536078956052761$	\\
$5$	& $11.75$	& $0.710918$& $20$& $8496.52466583252$	& $0.6526328570448198$	\\
$6$	& $18$		& $0.694988$& $21$& $13186.5479431152$	& $0.6517513959927006$	\\
$7$	& $28.125$	& $0.687683$& $22$& $20465.4138507843$	& $0.6509500100339601$	\\
$8$	& $43.5$	& $0.6803679$& $23$& $31762.1571750641$	& $0.6502183164425649$	\\
$9$	& $67.578125$	& $0.675387157$& $24$& $49294.4703259468$	& $0.6495474249565217$	\\
$10$& $104.703125$	& $0.6710160692$& $25$& $76504.4521541595$	& $0.6489302434805050$  \\
$11$& $162.6171875$	& $0.66775781255$& $26$& $118734.071074605$	& $0.6483605558312379$  \\
$12$& $252.39453125$	& $0.664961403407$& $27$& $184274.078072488$	& $0.6478330968533019$\\ 
$13$& $391.8349609375$	& $0.66262325504703$& $28$& $285991.447942734$	& $0.6473433028892230$\\
$14$& $608.00634765625$	& $0.660567326816035$& $29$& $443855.720274746$	& $0.6468872849936371$\\ 
$15$& $943.6259765625$	& $0.65880475493734$& $30$& $688859.379193604$	& $0.6464616660609581$\\ 
\end{tabular}
\end{center}
}
\vspace{0.1cm}
\caption{Values of $\|{\bf R}_n\|_{\rm abs}$ and $\log_{2^n}\hspace{-0.1cm}\big(\|{\bf R}_n\|_{\rm abs}\big)$ for $n=1,2,\ldots,30$.}
\label{table:values}
\end{table}

\section{Further remarks}\label{sec:further}

The question of the optimal value, or even the existence of such an optimal value for which \eqref{eq:limit} holds remains open. In this paper, we've shown  that, if this value exists, it is in the interval $[0.6274882485,0.646461661]$. 

A possible attack on a better lower bound would be to study how the sign changes of $\eta$ interact with rotations; that is to prove a result concerning 
\[\sup_{\theta\in[0,1]}\,\biggl|\sum_{m\leqslant x}\mathrm{e}^{2\pi \mathrm{i} m \theta} \eta(m) \biggr|.\] Such a theorem is known for the Thue--Morse series; there is a $C>0$ such that \[\sup_{\theta\in[0,1]}\,\biggl|\sum_{k\leqslant x}\mathrm{e}^{2\pi \mathrm{i} k \theta} t(k) \biggr| \, \leqslant \, C x^{\log_4(3)};\] see Gel'fond \cite{G1968}, Newman \cite{N1969}, and Newman and Slater \cite{NS1975}. The exponent $\log_4(3)$, here, has very recently been shown to be optimal by two of us (Coons and Maz\'a\v{c}, in progress). Though we are optimistic, it remains unclear if there is a connection between the optimal $\theta$'s for the Thue--Morse sequence and its autocorrelations.

While we believe the optimal exponent $\alpha$ for which \eqref{eq:limit} holds is worthy of interest, the question of interpretation of this quantity is immediate. Such a limit statement is quite common in the dimensional analysis of spectral measures, and this quantity could certainly be considered as some type of dimension related to $\eta$ or the spectral measure of the Thue--Morse sequence. Indeed, the so-called {\em correlation dimension} of the spectral measure of the Thue--Morse sequence is \[D_2:=1- \lim_{n\to\infty}\frac{1}{n}\log_2\hspace{-0.1cm}\Bigg(\sum_{m\leqslant 2^n}\eta^2(m)\Bigg)=1-\log_2((1+\sqrt{17})/4)\approx0.64298136.\] This is precisely the result of Zaks, Pikovsky, and Kurths \cite{ZPK1997} alluded to in the introduction which gives a lower bound on $\alpha$. Other powers give other generalized dimensions (so-called R\'enyi dimensions \cite{Renyi}); e.g., the power $0$ gives $D_0$---sometimes called the {\em capacity}. However, the limit $D_1$, corresponding to the first power, is not the limit we consider. This is the {\em information dimension}, which is determined from a Dirichlet series; \[D_1=2+\frac{2}{\log(2)}\sum_{m\geqslant 1}\frac{\eta(m)}{m}\approx 0.50638399544731967430.\] See Baake, Gohlke, Kesseb\"ohmer, and Schindler \cite[Remark 9.4]{BGKS2019} for a relevant discussion.

\subsection*{Acknowledgements}

It is our pleasure to thank Michael Baake, Philipp Gohlke, Fabian Gundlach, and Neil Ma\~nibo for helpful comments and conversations. The work of MC was supported by a David W.~and Helen E.~F.~Lantis Endowment, the work of MC, APK, and AS was supported by the National Science Foundation under DMS-2244020, and the work of JM was supported by the German Research Council (Deutsche Forschungsgemeinschaft, DFG) under CRC 1283/2 (2021 - 317210226). JM~thanks CSU Chico for their hospitality during his research visit in February~2023.

\bibliographystyle{amsplain}

\end{document}